\documentclass[twoside,12pt]{article}
\usepackage[T1]{fontenc}
\usepackage[utf8]{inputenc}
\usepackage{mathrsfs}
\usepackage{amssymb,amsmath,amsthm,mathtools}
\usepackage{graphicx}
\usepackage{color}
\usepackage[top=2cm,bottom=2cm,left=2cm,right=2cm]{geometry}
\usepackage{float,caption}
\usepackage{microtype}

\usepackage[colorlinks=true,
            linkcolor=blue,
            citecolor=blue,
            urlcolor=blue]{hyperref}

\input{epsf}

\newcommand{\bd}{\begin{description}}
\newcommand{\ed}{\end{description}}
\newcommand{\bi}{\begin{itemize}}
\newcommand{\ei}{\end{itemize}}
\newcommand{\be}{\begin{enumerate}}
\newcommand{\ee}{\end{enumerate}}
\newcommand{\beq}{\begin{equation}}
\newcommand{\eeq}{\end{equation}}
\newcommand{\beqs}{\begin{eqnarray*}}
\newcommand{\eeqs}{\end{eqnarray*}}

\definecolor{DarkGreen}{rgb}{0.2,0.6,0.3}

\newtheorem{theorem}{Theorem}[section]
\newtheorem{conjecture}[theorem]{Conjecture}

\newtheorem{lemma}[theorem]{Lemma}

\newtheorem{case}{Case}

\newtheorem{remark}{Remark}[section]

\newtheorem{proposition}[theorem]{Proposition}

\allowdisplaybreaks
\newcommand{\inv}{\operatorname{inv}}
\newcommand{\FF}{\mathbb{F}}
\newcommand{\eps}{\varepsilon}
\newcommand{\one}{\mathbf{1}}
\newcommand{\rk}{\operatorname{rank}}
\newcommand{\calH}{\mathcal{H}}

\title{\textbf{The inversion number of a path-reversed tournament: Resolving a conjecture of Belkhechine, Bouaziz, Boudabbous, and Pouzet}}

\author{Yaping Mao\footnote{Corresponding author. Academy of Plateau Science and Sustainability and School of Mathematics and Statistics, Qinghai Normal University, Xining, Qinghai 810008, China. {\tt yapingmao@outlook.com; myp@qhnu.edu.cn}}}
\date{}

\begin{document}

\maketitle

\begin{abstract}
Let $D$ be a tournament and let $X\subseteq V(D)$. The inversion of $X$ reverses all arcs whose
both endpoints lie in $X$ and leaves every other arc unchanged. A family of inversions is a
decycling family if applying all of them produces an acyclic, equivalently transitive,
tournament. The inversion number $\inv(D)$ is the minimum size of such a family.
Let $Q_n$ be the
tournament on $[n]$ obtained from the natural transitive tournament by reversing precisely the
consecutive pairs $12,23,\ldots,(n-1)n$. 
Belkhechine, Bouaziz, Boudabbous, and Pouzet conjectured in their
unpublished manuscript that a natural path-reversed family has
inversion number exactly $\left\lfloor(n-1)/2\right\rfloor$. The same problem was later recorded by
Bang-Jensen, da Silva, and Havet and by Alon, Powierski, Savery, Scott, and Wilmer.
In this paper we resolve this conjecture.\\[0.2cm]
\textbf{2020 Mathematics Subject Classification.} 05C20; 05C35.\\[0.2cm]
\textbf{Keywords.} Tournament; Inversion number; Transitive tournament; Decycling family; Finite fields.
\end{abstract}

\section{Introduction}

A \emph{tournament} is an
orientation of a complete graph. A tournament is \emph{transitive} if its vertices admit a linear
order in which every arc points forward. Equivalently, a tournament is transitive if and only if it
is acyclic. We write $ij$ for the unordered pair $\{i,j\}$; its orientation is one of $i\to j$ and
$j\to i$. For a statement $P$, the symbol $\one_P$ denotes its indicator, equal to $1$ if $P$ holds
and $0$ otherwise.

Let $D$ be a tournament and let $X\subseteq V(D)$. The \emph{inversion} of $X$ reverses all arcs whose
both endpoints lie in $X$ and leaves every other arc unchanged. A family of inversions is a
\emph{decycling family} if applying all of them produces an acyclic, equivalently transitive,
tournament. The \emph{inversion number} $\inv(D)$ is the minimum size of such a family.

Inversions of vertex sets provide a natural way to measure how far an oriented graph is from being
acyclic. For tournaments, acyclic is equivalent
to transitive, so the parameter measures the number of induced subtournament reversals needed to
produce a linear order.

The study of inversions goes back to Belkhechine's thesis and subsequent work of Belkhechine,
Bouaziz, Boudabbous, and Pouzet \cite{BelkhechineThesis,BBBPUnpublished,BelkhechineEtAl}.  A systematic
modern treatment for oriented graphs was given by Bang-Jensen, da Silva, and Havet
\cite{BangJensenDaSilvaHavet}, who related inversion number to cycle transversals and cycle
arc-transversals, studied complexity questions, and posed several structural problems.  Alon,
Powierski, Savery, Scott, and Wilmer \cite{APSSW} developed the tournament side further; among other
results, they proved fixed-parameter tractability for deciding whether a tournament has inversion number
at most $k$, proved NP-completeness for the corresponding problem on general oriented graphs, and
determined the asymptotic maximum inversion number of an $n$-vertex tournament to be $(1+o(1))n$.
These works show that the global theory is now quite well developed, while exact values for concrete
structured tournaments remain delicate.

Several recent lines of work have expanded the subject.  Aubian, Havet, H\"orsch, Klingelhoefer, Nisse,
Rambaud, and Vermande \cite{AubianEtAl} sharpened general extremal bounds and produced counterexamples
to natural dijoin additivity conjectures.  Behague, Johnston, Morrison, and Ogden \cite{BehagueEtAl}
connected inversion number with subgraph complementation, giving a useful linear-algebraic viewpoint.
Yuster \cite{Yuster} studied restricted inversions in which only sets of bounded size may be inverted.
Duron, Havet, H\"orsch, and Rambaud \cite{DuronEtAl} considered variants where the target is high
arc-connectivity or strong connectivity rather than acyclicity.  More recently, Ai \cite{Ai} initiated
a probabilistic direction by studying the inversion walk on labelled tournaments.

Let $Q_n$ be obtained from the natural transitive tournament on $[n]=\{1,\ldots,n\}$ by reversing exactly
the consecutive pairs
  $12,23,\ldots,(n-1)n$.
Equivalently, for $i<j$ we orient
\[
  i\to j \quad\text{if } j\ge i+2,
  \qquad
  i+1\to i \quad\text{if } j=i+1.
\]
Thus all nonconsecutive pairs keep their natural orientation, while each consecutive pair is oriented
backwards. In 2012, Belkhechine, Bouaziz, Boudabbous, and Pouzet conjectured in their unpublished manuscript that
\begin{conjecture}[Belkhechine, Bouaziz, Boudabbous, and Pouzet, \cite{BBBPUnpublished}]
\[
  \inv(Q_n)=\left\lfloor\frac{n-1}{2}\right\rfloor .
\]
\end{conjecture}
It was later
recorded by Bang-Jensen, da Silva, and Havet as Conjecture~1.11 and by Alon, Powierski, Savery,
Scott, and Wilmer as Conjecture~8.6 \cite[Conjecture~1.11]{BangJensenDaSilvaHavet} and \cite[Conjecture~8.6]{APSSW}.
The latter paper also notes that the conjecture was known for $n\le8$ and that the natural matching
construction gives the upper bound \cite[Conjecture~8.6]{APSSW}.  

Our main result proves that this
upper bound is always optimal.

\begin{theorem}\label{thm:main}
For every $n\ge1$, we have
\[
  \inv(Q_n)=\left\lfloor\frac{n-1}{2}\right\rfloor .
\]
\end{theorem}

\section{The upper bound}

We prove the upper bound by constructing an family of inversions.

\begin{theorem}\label{th-1}
For every $n\ge1$, we have
\[
  \inv(Q_n) \le \left\lfloor\frac{n-1}{2}\right\rfloor .
\]    
\end{theorem}

\begin{proof}
The vertices of $Q_n$ are labelled by
$1,2,\ldots,n.$
We call $(1,2,\ldots,n)$
the natural order of the labels. 
Set $m=\lfloor\frac{n-1}{2}\rfloor, n \geq 3$
and invert the $m$ two-vertex sets
$\{2,3\},\{4,5\},\ldots,\{2m,2m+1\}.$
Since each inverted set contains only two vertices, the inversion of
$\{2t,2t+1\}$ reverses only the edge between $2t$ and $2t+1$ for $1 \leq t \leq m$.

We now specify the order that will be the transitive order $\tau$ of the tournament
after these inversions. This order is not the original natural order of the
labels. Let
\[
\tau= \begin{cases} (2,1,4,3,\ldots,2b,2b-1), & \text{if } n=2b,\\[4pt] (2,1,4,3,\ldots,2b,2b-1,2b+1), & \text{if } n=2b+1. \end{cases} 
\]
Then $m=b-1$ for $n=2b$ and $m=b$ for $n=2b+1$. Thus, in $\tau$, each pair of consecutive labels $\{2r-1,2r\}$ appears in the order $(2r,2r-1),$
where $1\le r\le b.$
These two-vertex blocks are arranged in increasing order of their labels. In
the odd case $n=2b+1$, the remaining vertex $2b+1$ is placed at the end.

Consider first two vertices in the same block $(2r,2r-1)$ for $1 \le r \le b$. Their labels are
consecutive, so in the original tournament $Q_n$ the edge is $2r\to 2r-1.$
The pair $\{2r-1,2r\}$ is not one of the sets that we invert. Hence this edge is
unchanged. 

Next, consider two vertices $u$ and $v$ lying in different blocks with $u<v$, and suppose
that their labels are not consecutive. Then $v\ge u+2,$
so the edge in $Q_n$ has the natural orientation $u\to v.$
This edge is not affected by any inversion, because every inversion involves
only one of the pairs $\{2t,2t+1\}$ for $1 \leq t \leq m$.

It remains to consider consecutive labels belonging to different blocks.
These pairs are exactly $\{2r, 2r+1\},$ for $ 1\le r\le m.$
In the original tournament $Q_n$, their edge is $2r+1\to 2r,$
because consecutive edges are reversed relative to the natural order.
However, $\{2r,2r+1\}$ is one of the inverted sets, so this edge is changed to
$2r\to 2r+1$ in $\tau$.

These cases cover every pair of vertices. Hence, after the $m$ inversions, all
edges point forward in the order $\tau$. Therefore the resulting tournament is
transitive, and $\inv (Q_n)\le m
=\left\lfloor\frac{n-1}{2}\right\rfloor.$

For $n=1$ and $n=2$, we have $m=0$. In these two cases, $Q_n$ is already
transitive, so no inversion is required.    
\end{proof}

\section{The lower bound}

\subsection{Parity vector for an arbitrary inversion family}

All vector spaces, inner products, and ranks in the proof are over $\FF_2=\{0,1\}$ unless explicitly stated.
We now prove the lower bound. Fix an arbitrary final transitive order
\[
  \sigma=(s_1,s_2,\ldots,s_n)
\]
of the vertices. We shall prove a lower bound for inversion families whose final transitive order is
this prescribed order, and only then minimize over $\sigma$.

Suppose that $k$ inversions $X_1,\ldots,X_k$ transform $Q_n$ into the transitive tournament with order
$\sigma$. For each vertex $i\in[n]$, define its membership vector
\[
  x_i=(x_i(1),\ldots,x_i(k))\in\FF_2^k,
  \qquad
  x_i(r)=1 \Longleftrightarrow i\in X_r.
\]
For distinct vertices $i,j$, the edge $ij$ is flipped by the $r$th inversion exactly when both endpoints
belong to $X_r$, which is exactly the condition $x_i(r)x_j(r)=1$.  Hence the total parity with which
$ij$ is flipped is
\begin{equation}\label{eq:flip-parity}
  x_i\cdot x_j
  =\sum_{r=1}^k x_i(r)x_j(r)
  \quad\text{(mod} \ 2).
\end{equation}
Only this parity matters, because flipping the same edge twice restores its direction.  Equivalently,
inversions commute, so the order in which an inversion family is applied is irrelevant.

For a vertex $i$, let $\operatorname{pos}_\sigma(i)$ denote the position of $i$ in the order
$\sigma$. For distinct $i,j$, define $\eps_\sigma(i,j)=1$ if $i$ and $j$ have opposite relative order
in the natural order and in $\sigma$, and define $\eps_\sigma(i,j)=0$ otherwise.  Equivalently, for
$i<j$,
\[
  \eps_\sigma(i,j)=1
  \quad\Longleftrightarrow\quad
  \operatorname{pos}_\sigma(i)>\operatorname{pos}_\sigma(j),
\]
and we extend symmetrically by $\eps_\sigma(i,j)=\eps_\sigma(j,i)$.

\begin{lemma}\label{lem:basic}
For all distinct $i,j\in[n]$, we have
\begin{equation}\label{eq:basic}
  x_i\cdot x_j
  =\eps_\sigma(i,j)+\one_{\{|i-j|=1\}}
  \quad\text{(mod}\ 2).
\end{equation}
\end{lemma}

\begin{proof} 
Fix distinct vertices $i,j\in[n]$ of $Q_n$. Since both sides of
\eqref{eq:basic} are symmetric in $i$ and $j$, we may assume that $i<j.$
We consider the following two cases.

\begin{case}
$j\ge i+2$.
\end{case}

Then $i$ and $j$ are not consecutive, so their edge in $Q_n$ has the natural
orientation $i\to j.$
If $i$ appears before $j$ in $\sigma$, then the final orientation is also
$i\to j$. Hence the edge is reversed an even of times, and
\eqref{eq:flip-parity} gives $x_i\cdot x_j=0.$
Moreover, $\eps_\sigma(i,j)=0$ and 
$\one_{\{|i-j|=1\}}=0.$
Therefore, $x_i\cdot x_j
=\eps_\sigma(i,j)+\one_{\{|i-j|=1\}}.$

If $j$ appears before $i$ in $\sigma$, then the final orientation is
$j\to i$. Hence the edge is reversed an odd of times, and
\eqref{eq:flip-parity} gives $x_i\cdot x_j=1.$
In this case, $\eps_\sigma(i,j)=1$ and
$\one_{\{|i-j|=1\}}=0,$
so again $x_i\cdot x_j
=\eps_\sigma(i,j)+\one_{\{|i-j|=1\}}.$

\begin{case}
$j=i+1$.
\end{case}

Now $i$ and $j$ are consecutive, so their edge in $Q_n$ is $j\to i.$
If $i$ appears before $j$ in $\sigma$, then the final orientation is
$i\to j$. Thus the edge is reversed an odd of times, and
\eqref{eq:flip-parity} gives $x_i\cdot x_j=1.$
Furthermore, $\eps_\sigma(i,j)=0$ and
$\one_{\{|i-j|=1\}}=1,$
and hence
$x_i\cdot x_j
=\eps_\sigma(i,j)+\one_{\{|i-j|=1\}}.$

If $j$ appears before $i$ in $\sigma$, then the final orientation remains
$j\to i$. Thus the edge is reversed an even of times, and
\eqref{eq:flip-parity} gives $x_i\cdot x_j=0.$
Here,
$\eps_\sigma(i,j)=1$ and
$\one_{\{|i-j|=1\}}=1.$
Since the two terms on the right-hand side are both equal to $1$, it follows that their sum
vanishes modulo $2$. Hence
$x_i\cdot x_j
\equiv \eps_\sigma(i,j)+\one_{\{|i-j|=1\}}
\pmod{2}.$

Therefore, the equality holds in all cases, and \eqref{eq:basic} follows.
\end{proof}

The rest of the proof extracts many independent consequences of \eqref{eq:basic}.  Introduce an
auxiliary symbol $0$ by setting
$s_0=0, x_0=\{0,\ldots,0\}.$
The symbol $0$ is not a vertex of $Q_n$; it is used only to give the first element $s_1$ a formal
predecessor. In particular, no expression $\eps_\sigma(0,t)$ will be used.

For $v=s_r$, define $q(v)=s_{r-1}$
the predecessor of $v$ in the augmented sequence $0,s_1,\ldots,s_n$.  Thus $q(s_1)=0$ and
$q(v)\in[n]$ for all other vertices.  Also set
\[
  y_v=x_v+x_{q(v)},
  \qquad
  h(v)=\max\{v,q(v)\}.
\]
When $h(v)<n$, the vector $x_{h(v)+1}$ is one of the genuine vertex vectors, so all inner products
below are well-defined.

\subsection{The predecessor-difference pivot and fixed-order lower bound}

\begin{lemma}\label{lem:pivot}
Let $v\in[n]$. If $h(v)<n$ and $t >h(v)$, then
\begin{equation}\label{eq:pivot-one}
y_v\cdot x_t=
\begin{cases}
1, & \text{if } t=h(v)+1,\\[2mm]
0, & \text{if } t>h(v)+1.
\end{cases}
\end{equation}
\end{lemma}

\begin{proof}
Write $q=q(v)$ and $h=h(v)=\max\{q,v\}.$
We consider the two
possible values of $q$.

\setcounter{case}{0}
\begin{case}
$q\neq 0$.
\end{case}

By the definition of $q(v)$, the vertex $q$ occurs immediately before $v$ in
the order $\sigma$. Let $t>h$. Since $h=\max\{q,v\},$
we have $q<t$ and $v<t$ in the natural label order. Because $q$ and $v$ are consecutive in $\sigma$, the vertex $t$ cannot occur
between them. Hence, in $\sigma$, the vertex $t$ appears either before both
$q$ and $v$, or after both of them. Therefore the pairs $(q,t)$ and $(v,t)$
have the same relative-order status, and thus
\begin{equation}\label{eq:eps-cancel}
\eps_\sigma(q,t)=\eps_\sigma(v,t).
\end{equation}

Using $y_v=x_v+x_q$ and Lemma~\ref{lem:basic}, we obtain
\begin{align}
y_v\cdot x_t
&=(x_v+x_q)\cdot x_t \notag\\
&=x_v\cdot x_t+x_q\cdot x_t \notag\\
&=\eps_\sigma(v,t)+\one_{\{|v-t|=1\}}
  +\eps_\sigma(q,t)+\one_{\{|q-t|=1\}} \notag\\
&\equiv \one_{\{|v-t|=1\}}
  +\one_{\{|q-t|=1\}}
  \pmod{2}.
\label{eq:tail-formula}
\end{align}

Now take $t=h+1.$
Exactly one of $q$ and $v$ is equal to $h$, while the other is at most $h-1$.
Hence exactly one of the two pairs
$\{v,h+1\},\{q,h+1\}$
consists of consecutive labels. Therefore, exactly one indicator in
\eqref{eq:tail-formula} is equal to $1$, and
$y_v\cdot x_{h+1}=1.$

If instead $t>h+1$, then $q\le h$ and $v\le h$, and hence $t-q>1$ and $t-v>1.$
Therefore, $\one_{\{|q-t|=1\}}=0$ and 
$\one_{\{|v-t|=1\}}=0.$
It follows from \eqref{eq:tail-formula} that
$y_v\cdot x_t=0.$

\begin{case}
$q=0$.
\end{case}
Then $v=s_1$, so $v$ is the first vertex in the order $\sigma$. Moreover, $x_q=x_0, y_v=x_v+x_0=x_v$ and $h(v)=v.$ Since $t>h(v)=v$ in the natural label order and $v$ also appears before
$t$ in $\sigma$, the relative order of $v$ and $t$ is unchanged. Therefore, $\eps_\sigma(v,t)=0.$
Applying Lemma~\ref{lem:basic}, we obtain
$y_v\cdot x_t
=x_v\cdot x_t
=\one_{\{|v-t|=1\}}.$

If $t=v+1$, then $\one_{\{|v-t|=1\}}=1,$ and hence $y_v\cdot x_t=1.$
If $t>v+1$, then $v$ and $t$ are not consecutive, so $\one_{\{|v-t|=1\}}=0,$ and therefore, $y_v\cdot x_t=0.$
Thus \eqref{eq:pivot-one} also holds when $q=0$.
\end{proof}

For the fixed order $\sigma$, define the set of forced pivot labels
\[
  \calH(\sigma)=\{h(v):v\in[n],\ h(v)<n\}.
\]
Let $\inv_\sigma(Q_n)$ denote the minimum number of inversions that transform $Q_n$ into the transitive
tournament with order $\sigma$.

\begin{proposition}\label{prop:fixed-order}
For every order $\sigma$ of $[n]$, we have
$\inv_\sigma(Q_n)\ge |\calH(\sigma)|.$
\end{proposition}

\begin{proof}
Suppose that $k$ inversions transform $Q_n$ into the transitive tournament
with order $\sigma$. We prove that $k\ge |\calH(\sigma)|.$ Write the distinct elements of $\calH(\sigma)$ in increasing order as
$\calH(\sigma)=\{h_1<h_2<\cdots<h_a\}.$
Thus $a=|\calH(\sigma)|.$
If $a=0$, then the desired inequality is immediate. Assume therefore that
$a\ge 1$. For each $i\in\{1,\dots,a\}$, we select a vertex $v_i\in[n]$ such that $h(v_i)=h_i$. The existence of such a vertex follows directly from the definition of $\mathcal{H}(\sigma)$.

Now define the $a\times a$ matrix $C$ by
\begin{equation}\label{eq:C-def}
C_{ij}=y_{v_i}\cdot x_{h_j+1}.
\end{equation}
We first determine the triangular structure of $C$. Since $h(v_i)=h_i$, Lemma~\ref{lem:pivot} applied with $t=h(v_i)+1=h_i+1$
gives $C_{ii}
=y_{v_i}\cdot x_{h_i+1}
=1$ by \eqref{eq:C-def}.
Next, let $j>i$. Since $h_1<h_2<\cdots<h_a,$
we have $h_j+1>h_i+1.$
Applying the second part of Lemma~\ref{lem:pivot} to the vertex $v_i$, with
$t=h_j+1$, yields $C_{ij}
=y_{v_i}\cdot x_{h_j+1}
=0.$
Therefore, every entry above the main diagonal is zero. Thus $C$ is lower
triangular, with all diagonal entries equal to $1$.

Consequently, $C$ is nonsingular over $\FF_2$, and hence
\begin{equation}\label{eq:C-full-rank}
\rk_{\FF_2} C=a.
\end{equation}

We now obtain an upper bound on the rank of $C$. Let $Y$ be the $a\times k$
matrix whose $i$th row is $y_{v_i}$, and let $Z$ be the $a\times k$ matrix
whose $j$th row is $x_{h_j+1}$. Then the $(i,j)$-entry of $YZ^{\mathsf T}$ is $(YZ^{\mathsf T})_{ij}
=y_{v_i}\cdot x_{h_j+1}
=C_{ij}.$
Hence $C=YZ^{\mathsf T}.$
Since both $Y$ and $Z$ have $k$ columns, it follows that $\rk Y\le k$ and $\rk Z\le k.$
Therefore,
\[
\rk C
=\rk(YZ^{\mathsf T})
\le \min\{\rk Y,\rk Z\}
\le k.
\]
Combining this observation with \eqref{eq:C-full-rank}, we conclude that $a \le k.$
As $a = |\calH(\sigma)|$, this implies that any inversion family inducing the order $\sigma$ must contain at least $|\calH(\sigma)|$ inversions. We thus establish the lower bound $\inv_\sigma(Q_n) \ge |\calH(\sigma)|.$
\end{proof}

\begin{lemma}\label{lem:path-count}
For every order $\sigma=(s_1,\ldots,s_n)$ of $[n]$, we have $|\calH(\sigma)|
\ge \left\lfloor\frac{n-1}{2}\right\rfloor.$
\end{lemma}

\begin{proof}
For $n=1$, we have $\calH(\sigma)=\varnothing$, and hence both sides are
equal to $0$. Assume that $n\ge 2$.
Introduce the auxiliary symbol $s_0=0$ and consider the augmented path $s_0,s_1,\ldots,s_n.$
Its $n$ edges are $e_r=\{s_{r-1},s_r\}$ for $
r=1,\ldots,n.$
By the definition of $h(v)$, the larger endpoint of $e_r$ is $\max\{s_{r-1},s_r\}=h(s_r).$

Fix $j\in[n]$. Any path edge whose larger endpoint is $j$ must be incident
with $j$. Since every vertex of a path has degree at most $2$, the label $j$
can occur as the larger endpoint of at most two path edges. In particular, at most two edges have larger endpoint equal to $n$.
Therefore, at least $n-2$ of the $n$ edges have larger endpoint strictly smaller than $n$.

Let $E_{<n}=\bigl\{\{s_{r-1},s_r\}: h(s_r)<n\bigr\}.$
Since the label $n$ can be the larger endpoint of at most two path edges, it follows that $|E_{<n}|\ge n-2.$
Moreover, for each $h\in\calH(\sigma)$, at most two edges in $E_{<n}$ have
larger endpoint $h$. Therefore, 
$$
|E_{<n}|
\le \sum_{h\in\calH(\sigma)}2
=2|\calH(\sigma)|.
$$
Combining the two inequalities gives
$n-2\le 2|\calH(\sigma)|.$
Hence
\[
|\calH(\sigma)|
\ge
\left\lceil\frac{n-2}{2}\right\rceil
=
\left\lfloor\frac{n-1}{2}\right\rfloor.
\]

\end{proof}

\begin{proof}[Proof of Theorem \ref{thm:main}]
It follows from Proposition \ref{prop:fixed-order} and Lemma \ref{lem:path-count} that
$\inv(Q_n) \geq \left\lfloor\frac{n-1}{2}\right\rfloor.$
Combined with Theorem \ref{th-1}, this yields the exact equality $\inv(Q_n) = \left\lfloor\frac{n-1}{2}\right\rfloor.$
\end{proof}

\begin{remark}
The proof is non-computational. It does not enumerate orders, search for inversion families, or invoke a
rank-minimization theorem. The only algebraic device is the membership-vector encoding of an arbitrary
family of inversion sets over $\FF_2$, justified by the parity observation that inversions commute.  The
auxiliary symbol $0$ is used solely to define the predecessor of the first vertex in the final order; it is
never treated as a vertex of $Q_n$, and the case $q(v)=0$ is handled separately in Lemma~\ref{lem:pivot}.
Finally, the lower bound is not a covering argument: one inversion may flip many edges.  The obstruction is
linear independence. Different labels in $\calH(\sigma)$ yield different pivot columns, and the resulting
inner-product matrix is triangular with full rank.
\end{remark}

\section{Acknowledgement}

This work is supported by the National Science Foundation of China
(Nos. 12471329 and 12061059).

\end{document}